\documentclass[11pt,a4paper] {article} 
\usepackage{latexsym,amsmath,enumerate,amssymb,amsbsy,amsthm,enumerate,graphicx,color} 
 
\renewcommand\footnotemark{}
\date{}
 
\newtheorem{thm}{Theorem}[section]
\newtheorem{cor}[thm]{Corollary}
\newtheorem{lem}[thm]{Lemma}
\newtheorem{prop}[thm]{Proposition}
\theoremstyle{definition}

\newtheorem{rem}[thm]{Remark}
\newtheorem{ex}[thm]{Example}
\numberwithin{equation}{section}
\numberwithin{table}{section}
\numberwithin{figure}{section}

\begin{document}
\title{Self-Conjugate-Reciprocal Irreducible Monic Polynomials over Finite Fields} 
\author{{Arunwan Boripan},  {Somphong Jitman} and {Patanee Udomkavanich}
\thanks{A. Boripan is with the Department of Mathematics and Computer Science, Faculty of Science, Chulalongkorn University, Bangkok  10330, Thailand  (email: boripan-arunwan@hotmail.com).}
\thanks{S. Jitmann  is with the Department of Mathematics, Faculty of Science, Silpakorn University, Nakhon Pathom 73000, Thailand (email: sjitman@gmail.com).}
\thanks{P. Udomkavanich is with the Department of Mathematics and Computer Science, Faculty of Science, Chulalongkorn University, Bangkok  10330, Thailand (email: pattanee.u@chula.ac.th).}  
}

\maketitle 
\vspace{-0.5cm}
\begin{abstract}
 The class of  self-conjugate-reciprocal irreducible monic (SCRIM) polynomials  over finite fields are studied. Necessary and sufficient conditions for monic irreducible polynomials to be SCRIM are given. The number of SCRIM polynomials of a given degree  are also determined.
\end{abstract} 

\vspace{0.5cm}


 \section{Introduction}
 A polynomial  $f(x)$ of degree $n$ over a finite field $\mathbb{F}_q$ (with $f(0)\ne 0$) is said to be  \emph{self-reciprocal}   if $f(x)$ equals its  {\em reciprocal polynomial} $f^*(x):=x^nf(0)^{-1}f\left(\frac{1}{x}\right)$. A  polynomial is said to be  {\em self-reciprocal irreducible monic (SRIM)} if it is self-reciprocal, irreducible and monic.  SRIM and self-reciprocal polynomials  over finite fields have been studied and applied in various branches of Mathematics and Engineering. SRIM polynomials were used for characterizing and enumerating Euclidean self-dual cyclic codes over finite fields in \cite{JLX2011}   and for characterizing   Euclidean complementary dual cyclic codes over finite fields in \cite{YM}. In \cite{HB1975},   SRIM polynomials have been  characterized up to their degrees.    The order and the number  of  SRIM polynomials of a given degree  over finite fields have been determined in  \cite{YM2004}.

%
%
 
 In this paper, we focus on a generalization of a SRIM polynomial over finite fields, namely, a self-conjugate-reciprocal irreducible  monic  (SCRIM) polynomial. The {\em conjugate} of a  polynomial $f(x)=\sum_{i=0}^n f_i x^i$ over $\mathbb{F}_{q^2}$ is defined to be  $\overline{f(x)}=   \overline{f_0}+\overline{f_1}x+\dots+\overline{f_n}x^n$, where 
 $\bar{~} : \mathbb{F}_{q^2} \rightarrow  \mathbb{F}_{q^2}$ is defined by $\alpha\mapsto \alpha^{q}$ for all $\alpha\in \mathbb{F}_{q^2}$. A polynomial $f(x)$ over $\mathbb{F}_{q^2}$ (with $f(0)\ne 0$)  is said to be {\em self-conjugate-reciprocal} if $f(x)$  equals its  {\em conjugate-reciprocal polynomial} $f^\dagger(x):= \overline{f^*(x)}$.  
If, in addition, $f(x)$  is monic and irreducible, it is said to be \emph{self-conjugate-reciprocal irreducible monic (SCRIM)}.  SCRIM polynomials have been used for  characterizing Hermitian self-dual cyclic codes in \cite{SJ2014}. However,  properties of SCRIM polynomials have not been well studied. Therefore, it is of natural interest to characterize and to enumerate such polynomials.

\section{Preliminaries}

In this section, basic properties of polynomials that are important tools for studying SCRIM polynomials are recalled.

 Let  $q$ be a prime power and let $n$ be a positive integer such that  $\gcd(n,q)=1.$   Let   ${\rm ord}_{n}(q^2)$ denote  the multitplicative order of $q^2$ modulo $n$.   For each $0\leq i<n$, the \emph{cyclotomic  coset of  $q$ modulo  $n$  containing $i$} is defined to be the set 
\begin{align*}
Cl_q(i)=\{iq^j \, \textrm{mod n} \mid j \in \mathbb{N}_0\}.
\end{align*}
A {\em minimal polynomial} of an element $\alpha \in \mathbb{F}_{q^m}$ with respect to $\mathbb{F}_q$ is a nonzero monic polynomial $f(x)$ of least degree in $\mathbb{F}_q[x]$ such that $f(\alpha)=0$.

\bigskip

\begin{thm}[{\cite[Theorem 3.48]{LC2004}}] \label{TT}
Let $n \in \mathbb{N}$ be such that $\gcd(n,q)=1.$ Let $m \in \mathbb{N}$ satisfying $n\vert (q^m-1)$ and $\alpha$ be a primitive element of $\mathbb{F}_{q^m}$. Then 
\begin{align*}
M^{(i)}_{\mathbb{F}_{q}}(x)=\prod_{j\in Cl_q(i)}(x-\alpha ^j)
\end{align*}
is the minimal polynomial of $\alpha ^i$.
\end{thm}
\bigskip
\begin{rem} The polynomial $M^{(i)}_{\mathbb{F}_{q}}(x)$ in Theorem {\ref{TT}} will be referred to as the minimal polynomial of $\alpha^i$ defined corresponding to $Cl_q(i)$. 
\end{rem}

The {\em order} of a polynomial $f(x)$, denoted by $\mathrm{ord}(f(x))$,   is defined to be the smallest positive integer $s$ such that $f(x)$ divides $x^s-1$.
\smallskip
\begin{rem}\label{mini}
It is well know that if $f(x)$ is an irreducible polynomial over $\mathbb{F}_q$, then $f(x)|(x^{ord(f(x))}-1)$. Moreover, we have  \[\displaystyle x^{ord(f(x))}-1=\prod_{i=1}^t{M^{(i)}_{\mathbb{F}_{q}}}(x)'\] where $t$ is the cardinality of a complete set of representatives of the cyclotomic cosets of $q$ modulo $ord(f(x))$ {\cite[Theorem 3.48]{LC2004}}. It follows that any irreducible polynomials over $\mathbb{F}_q$ can be viewed as $M^{(i)}_{\mathbb{F}_{q}}(x)$ for some $i$.
 \end{rem}

   The following property of the order   mentioned in \cite{YM2004} and \cite{RH1997} is helpfull. 
\bigskip
\begin{lem}[{\cite[Theorem 3.3]{RH1997}}] \label{rootorder}
If $f(x)$ is an irreducible polynomial of degree $n$ over $\mathbb{F}_{q}$, then $\mathrm{ord}(f(x))$ is the order of any root of $f(x)$ in the multiplicative group $\mathbb{F}^*_{q^{n}}$.
\end{lem}

%
%
%

\section{Self-Conjugate-Reciprocal  Irreducible Polynomials}

In this section, we study self-conjugate-reciprocal irreducible monic (SCRIM) polynomials over finite fields. Since a SCRIM polynomial is defined over a finite field whose order is a square, for notation simplicity,  we focus on polynomials in  $\mathbb{F}_{q^2}[x]$. We determine the orders and  the number of SCRIM polynomials of a given degree.

\begin{lem} \label{sameroot}
Let $\alpha$ be an element in an extension field of $\mathbb{F}_{q^2}$ and let $f(x)\in \mathbb{F}_{q^2}[x]$. Then $\alpha$ is a root of $f\left( x\right) $ if and only if  $\alpha^{-q}$ is a root of $f^{\dagger} (x)$.
\begin{proof}
Let $f(x)=a_0+a_1x+\dots+a_nx^n$. Then 
\begin{align}
f^\dagger(\alpha^{-q})
&=\alpha^{-qn}(a_0^q+\frac{a_1^q}{\alpha^{-q}}+\dots+\frac{a_n^q}{\alpha^{-qn}}) \notag\\
&=\alpha^{-qn}(a_0+a_1\alpha+\dots+a_n\alpha^n)^q  \notag\\
&=\alpha^{-qn}(f(\alpha))^q.\notag
\end{align}
Therefore, $\alpha$ is a root of $f\left( x\right) $ if and only if  $\alpha^{-q}$ is a root of $f^{\dagger} (x).$
\end{proof}
\end{lem}

Next lemma gives a necessary and sufficient condition for an irreducible polynomial to be SCRIM. By Remark {\ref{mini}}, it suffices to concentrate on  $M^{(i)}_{\mathbb{F}_{q^2}}(x)$.

\begin{lem} \label{thm -qi}
$M^{(i)}_{\mathbb{F}_{q^2}}(x)$ is self-conjugate-reciprocal if and only if $Cl_{q^2}(i)=Cl_{q^2}(-qi)$.

\begin{proof} Assume $M^{(i)}_{\mathbb{F}_{q^2}}(x)=M^{\dagger(i)}_{\mathbb{F}_{q^2}}(x)$.
Then $\alpha^i$ is a root of $M^{\dagger(i)}_{\mathbb{F}_{q^2}}(x)$.
Since $Cl_{q^2}(-qi)$ is a class corresponding to $M^{\dagger(i)}_{\mathbb{F}_{q^2}}(x)$, By Theorem {\ref{TT}},  $i\in Cl_{q^2}(-qi)$. 
Hence, \[Cl_{q^2}(i)=Ci_{q^2}(-qi).\]

Conversely, assume that $Cl_{q^2}(i)=Cl_{q^2}(-qi)$. Then
\begin{align*}
M^{(i)}_{\mathbb{F}_{q^2}}(x)&=\prod_{j\in Cl_{q^2}(i)}(x-\alpha^j) \\
& =\prod_{j\in Cl_{q^2}(-qi)}(x-\alpha^{j}) \\
&=\prod_{j\in Cl_{q^2}(i)}(x-\alpha^{-qj}).
\end{align*}
Since $\alpha^{-qj}$ is a root of $M^{(i)}_{\mathbb{F}_{q^2}}(x)$ for all $j\in Cl_{q^2}(i)$, it follows that $\alpha^j$ is a root of $M^{\dagger(i)}_{\mathbb{F}_{q^2}}(x)$ for all $j\in Cl_{q^2}(-qi)$.
Therefore, $M^{(i)}_{\mathbb{F}_{q^2}}(x)=M^{\dagger(i)}_{\mathbb{F}_{q^2}}(x)$ as desired. 
\end{proof}
\end{lem}

\begin{thm}
The degree of a SCRIM  polynomial must be odd.
\begin{proof}
Assume that $M^{(i)}_{\mathbb{F}_{q^2}}(x)$ has degree $t$. If $t=1$, then the degree of $M^{(i)}_{\mathbb{F}_{q^2}}(x)$ is odd.
Suppose $t\neq 1$. Then, by Lemma{ \ref{thm -qi}}, we have $Cl_{q^2}(i)=Cl_{q^2}(-qi)$ and $\vert Cl_{q^2}(i)\vert =t>1$. Then there exists $0\leq j<t$ such that
\begin{align*}
i&\equiv (-qi)q^{2j}\, (\mathrm{mod \,n}).
\end{align*}
It follows that
\begin{align*}
-qi&\equiv(-q)(-qi)q^{2j}\, (\mathrm{ mod \,n})  \\
&\equiv iq^{2j+2}\, (\mathrm{mod \,n}),
\end{align*}
and hence, 
\begin{align*}
i&\equiv iq^{2j+2}q^{2j}\, (\mathrm{mod \,n})\\
&\equiv iq^{2(2j+1)}\, (\mathrm{mod \,n}).
\end{align*}
Since $t= {\rm ord}_{{\rm ord}(i)}(q^2)$ is a divisor of  ${\rm ord}_{n}(q^2)$, it follows that
\begin{align*} 
t\vert (2j+1),
\end{align*} where  ${\rm ord}(i)$ denotes the additive order of $i$ in $\mathbb{Z}_n$.
Then  $2j+1=kt $ for some positive integer $k$.  Since $0\leq j<t$, we have $kt \leq 2j+1 <2t$. It follows that $k=1$, and hence, 
$t=2j+1$ which is odd.
\end{proof}
\end{thm}

Next, we determine the number of SCRIM polynomials of degree $1$.

\begin{prop} \label{number1}
There are $q+1$ SCRIM polynomials of degree $1$ over $\mathbb{F}_{q^2}$.
\end{prop} 
\begin{proof}
Let $f(x)$ be a polynomial over $\mathbb{F}_{q^2}$ of degree $1$. Then $f(x)=x+a$ for some $a \in \mathbb{F}_{q^2}$.
Thus  $f^{\dagger}(x)=x+a^{-q}$. The polynomial  $f(x)$ is  SCRIM if and only if  $a=a^{-q}$. Equivalently, $a^{q+1}=1$. 

Since $(q+1)\vert |\mathbb{F}^{*}_{q^2}|$ and $\mathbb{F}^{*}_{q^2}$ is a cyclic group, there exists a unique subgroup $H$ of order $q+1$ of $\mathbb{F}^{*}_{q^2}$. Clearly,  $a^{q+1}=1$ if and only if  $a\in H$. Hence, the number of SCRIM polynomials of degree $1$ over $\mathbb{F}_{q^2}$ is $q+1$.
\end{proof}

\begin{ex}
By proposition {\ref{number1}}, there are $6$ SCRIM polynomials of degree $1$ over $\mathbb{F}_{25}$. In order to list all of them, we assume that $\mathbb{F}^*_{25}= \langle \alpha \rangle$. It can be easily seen that $1^6=1=(\alpha^4)^6=(\alpha^8)^6=(\alpha^{12})^6=(\alpha^{16})^6=(\alpha^{20})^6$. 

Hence, all SCRIM polynomials of degree $1$ over $\mathbb{F}_{25}$ are   $x+1$, $x+\alpha^4$, $x+\alpha^8$,  $x+\alpha^{12}$, $x+\alpha^{16}$ and $x+\alpha^{20}$.
\end{ex}

From now on, we assume that the polynomials have odd  degree $n\geq 3$. We determine the number of SCRIM polynomials of degree $n\geq 3$ by using the orders of SCRIM polynomials of degree $n$ over $\mathbb{F}_{q^2}$. The following three lemmas are important tools for determining the order of SCRIM polynomial.
\begin{lem} 
[{\cite[Proposition 2]{YM2004}}] \label{divider}Suppose $a, r$ and $k$ are positive integers with r even. If $a$ divides $q^r-1$ and $a$ divides $q^k+1$, then $a$ divides $q^{r/{2^s}}+1$ for some positive integer $s$.
\end{lem}

\begin{lem}[{\cite[Proposition 1]{YM2004}}] \label{positive}
Let $a$  be a positive integers with $a>2$. If $m$ is the smallest positive integer such that $a$ divides $q^m+1$, then, for any positive integer $s$, the following statements hold.
\begin{enumerate}
\item $a$ divides $q^s+1$ if and only if $s$ is an odd multiple of $m$.
\item $a$ divides $q^s-1$ if and only if $s$ is an even multiple of $m$.
\end{enumerate} 
\end{lem}

 Let $D_n$ be the set of all positive divisors of $q^n+1$ which do not divide $q^k+1$ for all $0\leq k <n$.

\begin{prop}\label{order}
Let $f(x)$ be a SCRIM polynomial of degree $n$ over $\mathbb{F}_{q^2}$. Then $\mathrm{ord}(f(x))\in D_n$. Moreover, if $\alpha \in \mathbb{F}_{q^{2n}}$ is a root of $f(x)$, then $\alpha$ is a primitive $d$-$th$ root of unity for some $d \in D_n$.
\end{prop}
\begin{proof}
Let $\alpha \in \mathbb{F}_{q^{2n}}$ be a root of $f(x)$. Since $f(x)$ is SCRIM, by Lemma \ref{sameroot}, $f(\dfrac{1}{\alpha^q})=0$ and 
we may write $\dfrac{1}{ \alpha^q}=\alpha^{q^{2t}}$ for some positive integer $t$. Then $\alpha^{q^{2t}+q}=1$ and thus $\mathrm{ord}(\alpha)$ divides $q^{2t}+q$. Since $\gcd(q,\mathrm{ord}(\alpha))=1$, we have $\mathrm{ord}(\alpha)\vert (q^{2t-1}+1)$. From $\alpha \in \mathbb{F}_{q^{2n}}$, then $\mathrm{ord}(\alpha)$ divides $q^{2n}-1$. By Lemma \ref{divider}, we have that $\mathrm{ord}(\alpha)$ divides $q^{2n/2^s}+1$ for some positive integer $s$. Since $n$ is odd, it follows that $s=1$. Then $\mathrm{ord}(\alpha)|(q^n+1)$.

Let $t$ be  the smallest  nonegative integer such that $\mathrm{ord}(\alpha)|(q^t+1)$.  Since $\deg (f(x))\geq 3$, we have $\mathrm{ord}(\alpha)\geq 3$, and hence, $t\geq 1$.   By Lemma \ref{positive},  $n$ is an odd multiple of $t$.  Using arguments similar to those in the proof of  {\cite[Proposition 4]{YM2004}}, we have  
$n=t$.  Therefore,  $\mathrm{ord}(\alpha)\nmid (q^k+1)$ for all $0\leq k <n$.  Hence, by Lemma {\ref{rootorder}}, $\mathrm{ord}(f(x))=\mathrm{ord}(\alpha)\in D_n$. From this, it can implies that $\alpha$ is a primitive $d$-$th$ root of unity for some $d \in D_n$.
\end{proof}

\begin{cor} \label{orderf}  
 Let $f(x)$ be a SCRIM polynomial of degree $n$ over $\mathbb{F}_{q^{2}}$. If $\alpha$ is a primitive element of $\mathbb{F}_{q^{2n}}$ and $\alpha ^j$ is a root of $f(x)$, then 
\begin{align*}
\mathrm{ord}(f(x))=\dfrac{q^{2n}-1}{\gcd(q^{2n}-1,j)}.
\end{align*}
\begin{proof} 
Let $\alpha$ be a primitive element of $\mathbb{F}_{q^{2n}}$ and let $\alpha ^j$ be  a root of $f(x)$. Then 
\begin{align*}
\mathrm{ord}(\alpha^j)=\dfrac{q^{2n}-1}{\gcd(q^{2n}-1,j)}.
\end{align*}
From Lemma \ref{rootorder}, we know that if $f(x)$ is irredeucible of degree $n$, then $\mathrm{ord}(f(x))$ is the order of any root of $f(x)$ in the multiplicative group $\mathbb{F}_{q^{2n}}^*$, so $\mathrm{ord}(\alpha^j)=\mathrm{ord}(f(x))$.
Hence, $\mathrm{ord}(f(x))=\mathrm{ord}(\alpha^j)=\dfrac{q^{2n}-1}{\gcd(q^{2n}-1,j)}$.
\end{proof}
\end{cor}


%
%

\begin{prop} \label{gcd}
If $d \in D_n$ and $\beta$ is a primitive $d$-$th$ root of unity, then the set $\{\beta, \beta^{q^2},..., \beta^{q^{2(n-1))}}\}$ is a collection of $n$ distinct primitive $d$-$th$ roots of unity.
\begin{proof}
Since $d \vert (q^n+1)$, we have $d\vert (q^{2n}-1)$. 
Let $0\leq i \leq n-1$. From $d\vert (q^{2n}-1)$, it follows that $\gcd(d,q^{2i})=1$ and $\beta^{q^{2i}}$ is a primitive $d$-$th$ root of unity. If $\beta^{q^{2i}}=\beta^{q^{2j}}$ for some $0\leq i < j \leq n-1,$ then $\beta^{q^{2i}-q^{2j}}=1$ so that $d$ divides $q^{2i}-q^{2j}=q^{2i}(q^{2(j-i)}-1)$. Since $\gcd(d,q^{2i})=1$, we see that $d$ divides $q^{2(j-i)}-1.$ Hence, by Lemma \ref{positive}, $2(j-i)=kn$ for some even positive integer $k$. But then $j=\dfrac{kn}{2}+i \geq n,$ a contradiction. Hence, $\beta^{q^{2i}}$'s are distinct.
\end{proof}
\end{prop}

Let $d\in D_n$ and let $\beta$ be a primitive $d$-$th$ root of unity over $\mathbb{F}_{q^2}$. Define the polynomial $f_\beta(x)=\displaystyle \prod_{i=0}^{n-1}(x-\beta^{q^{2i}})$.

\begin{prop} \label{conjugate}
 $f_\beta(x)$ is a SCRIM polynomials  of degree $n$ and order $d$.
\begin{proof}
Using the definition of $f_\beta(x)$ and the fact that $n$ is odd, we have
\begin{align}\label{sss}
\notag  f_{\beta}^{\dagger}(x)&= \displaystyle \prod_{i=0}^{n-1}(x-\beta^{q^{2i}}) ^{\dagger} \\ \notag 
&=\prod_{i=0}^{n-1}x(-\beta^{q^{2i}})^{-q}( \dfrac{1}{x}-{\beta}^{q^{2i+1}})\\  \notag
&=\prod_{i=0}^{n-1}(-\beta^{-q^{2i+1}}) \prod_{i=0}^{n-1}(1-\beta ^{q^{2i+1}}x)\\  \notag
&=\prod_{i=0}^{n-1}(-\beta^{-q^{2i+1}}) \prod_{i=0}^{n-1}\beta^{q^{2i+1}}\prod_{i=0}^{n-1}(\beta^{-q^{2i+1}}-x)\\  
&= \prod_{i=0}^{n-1}(x-\beta^{-q^{2i+1}}). 
\end{align}
We claim that $\{\beta^{q^{2j}}\mid 0\leq j\leq n-1\}=\{\beta^{-q^{2i+1}}\mid 0\leq i\leq n-1\}$. 

Let $\beta^{-q^{2s+1}} \in \{\beta^{-q^{2i+1}}\mid 0\leq i\leq n-1\}.$  Then
\begin{align*}
\beta^{-q^{2s+1}}&=\beta^{q^{2s}(-q)} =(\beta^{-q})^{q^{2s}} =({\beta^{q^{n+1}}})^{q^{2s}} ={\beta^{q^{n+1+2s}}}.
\end{align*}
Since $n$ is odd, we have  $\beta^{-q^{2s+1}}=\beta^{q^{2l}}$   for some  $0\leq l\leq n-1.$
Hence, $\beta^{-q^{2s+1}} \in \{\beta^{q^{2j}}\mid 0\leq j\leq n-1\}.$

Let  $\beta^{q^{2i}} \in \{\beta^{q^{2j}}\mid 0\leq j\leq n-1\} $. Since $n$ is odd, we have 
\begin{align*}
\beta^{q^{2i}}&= {\beta^{q^{n+1+2s}}}   = ({\beta^{q^{n+1}}})^{q^{2s}} = (\beta^{-q})^{q^{2s}} = \beta^{q^{2s}(-q)}
\end{align*}
for some  $0\leq s\leq n-1$.
Hence, $\beta^{q^{2i}} \in \{\beta^{-q^{2i+1}}\mid 0\leq i\leq n-1\}$. Therefore, $\{\beta^{q^{2j}}\mid 0\leq j\leq n-1\}=\{\beta^{-q^{2i+1}}\mid 0\leq i\leq n-1\}$ as desired.

From (\ref{sss})  and the fact that $\{\beta^{q^{2j}}\mid 0\leq j\leq n-1\}=\{\beta^{-q^{2i+1}}\mid 0\leq i\leq n-1\}$, we have
\begin{align*}
f_{\beta}^{\dagger}(x)
&= \prod_{i=0}^{n-1}(x-\beta^{-q^{2i+1}})\\
&= \prod_{j=0}^{n-1}(x-\beta^{q^{2j}})\\
&= f_\beta(x).
\end{align*}

Suppose that $f_\beta(x)$ is   witten as $f_\beta(x)=g(x)h(x)$, where $g(x)$ is  an irreducible monic polynomial of degree $r$ and $h(x)$ is  a monic polynomial of degree $n-r$. Let $\alpha$ be a root of $g(x)$.
Then
\begin{align*}
\alpha^{q^{2r}-1}=1.
\end{align*}
Since $\alpha$ is a root of $f_\beta (x)$, $\alpha$ is a $d$-$th$-root of unity.
Hence,
\begin{align*}
d|(q^{2r}-1).
\end{align*} 
Since $d$  divides $q^n+1$, by Lemma {\ref{positive}}, $2r$ is an even multiple of $n$. Since $r\leq n$, we have $r=n$ and $f_\beta(x)=g(x)$ is irreducible.
\end{proof}
\end{prop}

 The construction of a SCRIM polynomial $f_{\beta}(x)$ can be illustrated  as follows. 

\begin{ex}\label{ex} Let $n=3$ and $q=3$. Then $D_3=\{7,\,14,\,28\}$.  Assume that $\mathbb{F}_{729}^*=\langle \alpha \rangle$.  Since the set $\{ \alpha^{52},\,  \alpha^{468}, \,\alpha^{572}\}$ is a collection of $3$ distinct primitive $14$-$th$ roots of unity, it follows that
\begin{align*}
f_{\alpha^{52}}(x)= f_{\alpha^{468}}(x)= f_{\alpha^{572}}(x)=(x-\alpha^{52})(x-\alpha^{468})(x-\alpha^{572}).
\end{align*}
By Theorem {\ref{conjugate}}, $f_{\alpha^{52}}(x)$ is a SCRIM polynomial.                                             
\end{ex}

\begin{lem} [{\cite[Theorem 2.45]{RH1997}}]\label{cha}
Let $\mathbb{F}$ be a field of characteristic $p$ and let  $n$  be a positive integer not divisible by $p$. Let $\zeta$ be a primitive $d$-$th$ root of unity over $\mathbb{F}$. Then 
\begin{align}
x^n-1=\prod_{d\vert n}Q_d(x),
\end{align}
where 
$
Q_d(x)= \displaystyle\prod_{s=1 ,\ \gcd(s,n)=1}^n(x-\zeta ^s ).
$\end{lem}
Note that $Q_d(x)$ can be viewed as 
\begin{align*}
Q_d(x)= \displaystyle\prod_{\eta \in D}(x-\eta ),
\end{align*}
where $ D$ is the set of all primitive $ d$-$th$ roots of unity over  $\mathbb{F}$.

\begin{thm} \label{property}
Let $f(x)$ be  an irreducible monic polynomial of degree $n$ over $\mathbb{F}_{q^2}$. Then the following statements are equivalent:
\begin{enumerate}[$1)$]
\item $f(x)$ is self-conjugate-reciprocal.
\item $\mathrm{ord}(f(x)) \in D_n$.
\item $f(x)=f_\beta(x)$ for some primitive d-th root of unity $\beta$ with $d \in D_n$.
\end{enumerate}
\end{thm}
\begin{proof} By Corollary \ref{orderf} and Proposition \ref{conjugate}, it remains to prove $2)$ implies $3)$. Assume $\mathrm{ord}(f(x))\in D_n$. Let $p$ be the characteristic of $\mathbb{F}_{q^2}$. Since $\gcd(p, \mathrm{ord}(f(x)))=1$, by Lemma \ref{cha}, we have \[x^{\mathrm{ord}(f(x))}-1=\displaystyle \prod_{\ell \vert \mathrm{ord}(f(x))}Q_\ell(x).\] Since $f(x)\vert (x^{\mathrm{ord}(f(x))}-1)$, we have $f(x)\vert Q_d(x)$ for some divisor $d$ of $\mathrm{ord}(f(x))$. Then $d\vert (q^n+1)$.

We claim that $d\in D_n$. Suppose $d\vert (q^k+1)$ for some $k<n$. Then $d\vert (q^{2k}-1)$, ${\it i.e.}$, $q^{2k}\equiv 1\, (\mathrm{mod} \, d)$. From {\cite[Theorem 2.47]{RH1997}},  $n$ is the smallest positive integer such that $q^{2n}\equiv 1\,(\mathrm{mod}\,d)$. Since $k<n$, we have a contradiction. Therefore, $d\in D_n$.

Let $\gamma$ be a primitive $d$-$th$ root of unity over $\mathbb{F}_{q^2}.$ Since $q^{2n}\equiv 1\,(\mathrm{mod} \,d)$ and $q^{2k}\not\equiv 1 \, (\mathrm{mod} \,d)$, for all $0 \leq k<n$, it follows that $\gamma \in \mathbb{F}_{q^{2n}}$ but $\gamma \not\in \mathbb{F}_{q^{2k}}$ for all $0\leq k <n$. Then the minimal polynomial of $\gamma$ has degree $n$. Since $f(x)$ is irreducible  and $f(x)\vert Q_d(x)$, there exists a primitive $d$-$th$ root of unity $\delta$ such that its minimal polynomial equals $f(x)$. 

Finally, we show that $f(x)=f_\delta (x)$. Since  $f_\delta (x)$ and $f(x)$ are monic irreducible polynomials of the same degree $n$ and $\delta$ is a root of $f_\delta (x)$,  we have $f(x)=f_\delta (x)$.
\end{proof}


In the next theorem, we determine the number of SCRIM polynomials of a given degree. 

\begin{thm} Let $n\geq 3$ be an odd positive integer. Then following statements hold. 
\begin{enumerate}[$1)$]
\item For each $d\in D_n$, there are $\dfrac{\phi(d)}{n}$ SCRIM polynomials over $\mathbb{F}_{q^2}$ of degree $n$ and order $d$.
\item The number of SCRIM polynomials over $\mathbb{F}_{q^2}$ of degree $n$ is \[\dfrac{1}{n} \displaystyle\sum_{d\epsilon D_n}\phi(d).\]
\end{enumerate}
\end{thm}
\begin{proof}
For each $d\in D_n$, there are $\phi(d)$ primitive $d$-$th$ root of unity. For each primitive $d$-$th$ root of unity $\beta$, $f_\beta(x)$ has degree $n$ by \cite[Theorem 2.47]{RH1997}. Therefore, there are $\displaystyle \frac{\phi(d)}{n}$ SCRIM polynomials over $\mathbb{F}_{q^2}$ of degree $n$ and order $d$. Hence, $1)$ is proved.

Next, we show that $d=\mathrm{ord}(f_\beta(x))$. From the proof of Theorem {\ref{property}}, we know $d\leq \mathrm{ord}(f_\beta(x))$. Since $f_\beta(x)\vert Q_d(x)$, we have $f_{\beta}(x)\vert (x^d-1)$. It follows that $\mathrm{ord}(f_\beta(x))\leq d$. Hence,   $d=\mathrm{ord}(f_\beta(x))$.

The statement $2)$ follows  from $1)$ and the equivalence of the statements  $1)$ and $ 2)$ in Theorem {\ref{property}}.
\end{proof}

\begin{ex} \label{ex13}
Let $q=3$ and $n=3$. Then $D_3=\{7, \,14, \, 28\}$. Let $\alpha$ be defined as in Example {\ref{ex}}. Then, we have the following properties.
\begin{enumerate}[$1)$]
\item If $d=7$, there are $2$ SCRIM polynomials over $\mathbb{F}_{3^2}$ of degree $3$ and order $7$  which are $x^3 + a^3x^2 + a^5x + 2$, and $x^3 + ax^2 + a^7x + 2$.
\item If $d=14$, there are $2$ SCRIM polynomials over $\mathbb{F}_{3^2}$ of degree $3$ and order $14$  which are $x^3 + a^5x^2 + a^7x + 1$ and $x^3 + a^7x^2 + a^5x + 1$.
\item If $d=28$, there are $4$ SCRIM polynomials over $\mathbb{F}_{3^2}$ of degree $3$ and order $28$  which are  
$x^3 + ax^2 + ax + a^6$, $x^3 + a^3x^2 + a^3x + a^2$, $x^3 + a^5x^2 + ax + a^2$ and $x^3 + a^7x^2 + a^3x + a^6$.
\end{enumerate}
\end{ex}

Table $3.1$ displays the number of   SCRIM polynomials of   degree $n=1,3,7,\dots, 15$ over $\mathbb{F}_{q^2}$, where $q=2,3,5$.

\begin{table}[!htb]
\begin{center}
\begin{tabular}{|c|c|c|c|}
\hline
$q$&$n$& The number of SCRIM polynomials of degree $n$ over $\mathbb{F}_{q^2}$ \\ \hline
2&1&3\\
&3&2 \\
&5&6 \\

&7&18 \\

&9&56\\
&11 &186\\ 
&13&630 \\
&15&2182 \\\hline
3&1&4\\
&3&8\\
&5&48 \\
&7&312 \\
&9&2184\\
&11&16104\\ 
&13&122640\\
&15&956576 \\\hline
5&1&6\\
&3&40 \\
&5&624 \\
&7&1160 \\
&9&217000\\
&11&4438920\\ 
&13&93900240\\
&15&2034504992 \\\hline

\end{tabular}
\caption{The number of SCRIM polynomials of a given degree over $\mathbb{F}_{q^2}$.}
\end{center}
\end{table}

The order of SCRIM polynomials of degree  $n=11$ over $\mathbb{F}_4$ are listed in Table $3.2$ together with the number of SCRIM polynomials of degree  $n=11$ over $\mathbb{F}_4$  of a given order.

\begin{table}[!htb]
\begin{center}
\begin{tabular}{|c|c|}
\hline
Order & The number of SCRIM polynmials of each order \\ \hline
99&4\\
331&22\\
993&44\\
2979&132\\
3641&220\\
10928&440\\
32769&1320\\   \hline
Total& 2182 \\ \hline
\end{tabular}
\caption{The number of SCRIM polynomials of degree $15$ over $\mathbb{F}_4$.}
\end{center}
\end{table}

\section*{Acknowledgements}
  The authors would like to thank the anonymous referees for helpful comments and
suggestions.

\end{document}